\documentclass [11pt]{amsart}  
\usepackage{a4}                
\usepackage{paralist}           
\usepackage{graphicx}
\usepackage{amssymb}                        
\usepackage{float}
\usepackage{amsmath}
\usepackage{psfrag}
\usepackage{MnSymbol}
\usepackage[all]{xy}

\usepackage{amsmath,amsthm}

\theoremstyle{plain} 
\newtheorem{thm}{Theorem}[section]

\newtheorem{lem}[thm]{Lemma}

\newtheorem{prop}[thm]{Proposition}
\newtheorem{cor}[thm]{Corollary}

\theoremstyle{definition}
\newtheorem{defn}[thm]{Definition}

\newtheorem*{thmnn}{Theorem}

\newtheorem{ques}{Question}

\theoremstyle{remark}
\newtheorem{rem}[thm]{Remark}

\newtheoremstyle{TheoremNum}
        {\topsep}{\topsep}              
        {\itshape}                      
        {}                              
        {\bfseries}                     
        {.}                             
        { }                             
        {\thmname{#1}\thmnote{ \bfseries #3}}
    \theoremstyle{TheoremNum}
    \newtheorem{thmn}{Theorem}

\DeclareMathOperator{\kernel}{ker}

\DeclareMathOperator{\tor}{Tor}

\DeclareMathOperator{\torlen}{TorLen}

\DeclareMathOperator{\ord}{o}
\DeclareMathOperator{\stable}{stable}
\DeclareMathOperator{\relations}{relations}
\DeclareMathOperator{\out}{Out}
\DeclareMathOperator{\defi}{def}

\newcommand{\nsub}{\mathrel{\unlhd}}

\title{Torsion, torsion length and finitely presented groups}
\author{Maurice Chiodo and Rishi Vyas}

\begin{document}

\let\thefootnote\relax\footnotetext{2000 \textit{AMS Classification:} 20E06, 20F05, 20F06, 20F10, 20F65.}
\let\thefootnote\relax
\let\thefootnote\relax\footnotetext{The first author was partially supported by the Swiss National Science Foundation grant FN PP00P2-144681/1. This work is part of a project that has received funding from the European Union's Horizon 2020 research and innovation programme under the Marie Sk\l{}odowska-Curie grant agreement No.~659102.}
\let\thefootnote\relax\footnotetext{The second author was partially supported by Israel Science Foundation grants 170/12 and 253/13, the Center for Advanced Studies in Mathematics at Ben-Gurion University of the Negev, and by the Israel Council for Higher Education.}

\begin{abstract}
We show that a construction by Aanderaa and Cohen used in their proof of the Higman Embedding Theorem preserves torsion length. We give a new construction showing that every finitely presented group is the quotient of some $C'(1/6)$ finitely presented group by the subgroup generated by its torsion elements. We use these results to show there is a finitely presented group with infinite torsion length which is $C'(1/6)$, and thus word-hyperbolic and virtually torsion-free.
\end{abstract}

\maketitle

\section{Introduction}

It is well known that the set of torsion elements in a group $G$, $\tor(G)$, is not necessarily a subgroup. One can, of course, consider the subgroup $\tor_{1}(G)$ \emph{generated} by the set of torsion elements in $G$; this subgroup is always normal in $G$. 

The subgroup $\tor_{1}(G)$ has been studied in the literature, with a particular focus on its structure in the context of $1$-relator groups. For example, suppose $G$ is presented by a $1$-relator presentation $P$ with cyclically reduced relator $R^{k}$ where $R$ is not a proper power, and let $r$ denote the image of $R$ in $G$. Karrass, Magnus, and Solitar proved (\cite[Theorem 3]{KaMaSo})  that $r$ has order $k$ and that every torsion element in $G$ is a conjugate of some power of $r$; a more general statement can be found in \cite[Theorem 6]{Sch}. As a immediate corollaries, we see that $\tor_{1}(G)$ is precisely the normal closure of $r$, and that $G/\tor_{1}(G)$ is torsion-free. 

More generally, the manner in which $\tor_{1}(G)$ is impacted by the \emph{deficiency} $\defi(G)$ of a finitely presented group $G$ has also been investigated. The deficiency of $G$ is the maximum value of $m-n$, where $m$ and $n$ are the number of generators and relators respectively as we range over all finite presentations of $G$.  In \cite[Corollary 3.6]{BerHil}, Berrick and Hillman proved that if $G$ is a finitely presentable group with $\defi(G)>0$, and $\tor_{1}(G)$ is either finitely generated or locally finite, then $\tor_{1}(G)$ is actually finite; again, in this situation $G/\tor_{1}(G)$ is torsion-free.  They claim that the question of whether $\tor_{1}(G)$ is necessarily \emph{trivial} under these hypotheses is open; using a result of Karras and Solitar \cite[Main Theorem]{KarSol} one immediately sees that this triviality is indeed the case when $G$ is presented by a $1$-relator presentation.

In both cases described above, the quotient $G/\tor_{1}(G)$ is torsion-free. Unfortunately, this is not always the case. Consider, for example, the group $C$ presented by the following presentation: $\langle x, y, z\ \vert \  x^{3}=e, y^{3}=e, xy=z^{3} \rangle$; it can be shown that $C$ is a finitely presented word-hyperbolic group (\cite[Proposition 3.5]{ChiVya}), but that $C/\tor_{1}(C) \cong \mathbb{Z}/3\mathbb{Z}$ (\cite[Proposition 3.1]{ChiVya}). 

We can, however, iterate the process used to construct $\tor_{1}(G)$ to produce an ascending chain of normal subgroups $\tor_{1}(G)\leq \tor_{2}(G) \leq \ldots$ of $G$. For finite $n\in \mathbb{N}$, we define $\tor_{n+1}(G)$ via $\tor_{n+1}(G)/\tor_{n}G = \tor_{1}(G/\tor_{n}(G))$; we define $\tor_{\omega}(G):=\bigcup_{n \in \mathbb{N}}\tor_{n}(G)$. The ordinal for which this chain stabilises is called the \emph{torsion length} of $G$ and denoted by $\torlen(G)$. It turns out that $ G/\tor_{\omega}(G)$ is the universal torsion-free quotient of $G$: it is torsion-free, and all other torsion-free quotients uniquely factor through it. Thus $\torlen(G)$ is  always bounded above by $\omega$; this bound is attained when the chain mentioned above does not stabilise at any finite stage.  Intuitively, $\torlen(G)$ is the minimal number of times we need to `kill off' torsion to get a torsion-free quotient of $G$.
 
The notion of torsion length first appeared, independently, in both \cite{Cirio et al} and our earlier work \cite{ChiVya}. In \cite{Cirio et al}, Cirio \emph{et al.}~defined the \emph{torsion degree} of a quantum group (here, quantum groups are C*-algebras equipped with a suitable comultiplication).  The definition of torsion length aligns with torsion degree when a group is viewed as a quantum group via its associated C*-algebra.  Further, they defined the notion of the ``connected component at the identity" $Q^{\circ}$ of a quantum group $Q$ 
and remarked that for an ordinary group $G$ (again viewed as a quantum group via its associated C*-algebra) this object corresponds to $G/\tor_{\omega}(G)$ (\cite[Example 3.17]{Cirio et al}). They also constructed a descending ordinal indexed family of quantum subgroups $G_{\alpha}$ ``converging" to $G^{\circ}$; again, in the classical situation these objects correspond to the quotients $G/\tor_{\alpha}(G)$.

The quotient $G/\tor_{\omega}(G)$ was first studied in \cite{BrodHow}, where Brodsky and Howie investigated this object (they use the notation $\hat{G}$) for various families of groups. A group is \emph{locally indicable} if every non-trivial finitely generated subgroup admits a surjection onto $\mathbb{Z}$: Brodsky and Howie showed that if a group has deficiency $\defi(G)>0$, then $G/\tor_{\omega}(G)$ is locally indicable \cite[Theorem 3.7]{BrodHow}. They also showed that $G/\tor_{\omega}(G)$ is locally indicable when $G$ is $1$-relator, or $2$-relator with one relator having length $\leq 4$, or $2$-relator with with one relator having length $5$ and the other has length  $\leq 8$, or at most $5$-relator with each relator having length $\leq  3$. These results appear as \cite[Theorems 1.1--1.4]{BrodHow}.

In \cite{ChiVya} we began a preliminary investigation of torsion length. 
One of the main results of that work, which we generalise here in Theorem \ref{gen first paper}, was the following theorem:

\begin{thmnn}{\cite[Theorem 3.3]{ChiVya}}
There is a family of finitely presented groups $\{P_{n}\}_{n \in \mathbb{N}}$ such that:
\\$1$. $P_{n+1} / \tor_{1}(P_{n+1}) \cong P_{n}$,
\\$2$. $\torlen(P_{n})=n$.
\end{thmnn}

We then showed that a construction used to prove a classic embedding theorem of Higman, Neumann and Neumann (every countable group embeds into a 2 generator group) preserved torsion length. This fact, used with the theorem mentioned above, allowed us to arrive at the following result:

\begin{thmnn} {\cite[Theorem 3.10]{ChiVya}}
There exists a $2$-generator recursively presented group $Q$ for which $\torlen(Q) = \omega$.
\end{thmnn}

This paper aims to extend \cite[Theorem 3.10]{ChiVya}. In Theorem \ref{txfpg}, we prove the following:

\vspace{3mm}

\begin{thmn} 
There exists a finitely presented group $F$ with $\torlen(F) = \omega$.

\end{thmn}

\vspace{3mm}

 We do this by showing that a particular construction used in a proof of the Higman Embedding Theorem preserves this invariant. 

Let us be more precise. The Higman Embedding Theorem \cite{Hig emb} states that a finitely generated, recursively presented group embeds into a finitely presented group. There are many proofs of this result, but these arguments share a common theme: they are all constructive. One must begin with a finite generating set for the group, and an algorithm that computes its relations, and then explicitly build a finitely presented group from this data.  In this paper we pick a particular construction, due to Aanderaa and Cohen \cite{Cohen2, Cohen3} and presented in \cite{Cohen}, examine it in detail, and conclude that the torsion length of the finitely presented group so constructed is the same as that of the recursively presented group that we started with.

The existence of a finitely presented group with infinite torsion length is then an immediate consequence of \cite[Theorem 3.10]{ChiVya}: take the recursively presented group constructed in \textit{loc}.~\textit{cit}.~and apply the Aanderaa-Cohen construction. 

Section \ref{hyp1B} of this paper is concerned with improving Theorem \ref{txfpg}. The following result appears as  Theorem \ref{fp hyp vspecialB}.

\vspace{3mm}

\begin{thmn}
 There exists a finitely presented word-hyperbolic virtually special group $W$ with $\torlen(W) = \omega$. In particular, $W$ is virtually torsion-free.
\end{thmn}

\vspace{3mm}

This is done using small cancellation theory. Of particular importance is the following theorem, whose content is contained in Proposition \ref{powersB} and Theorem \ref{exact tor quot}; this result is also of independent interest.  Combined with  Theorem \ref{txfpg}, it proves  Theorem \ref{fp hyp vspecialB}.

\vspace{3mm}

\begin{thmn}
Let $P=\langle x_{1}, \ldots, x_{m} \ | \ r_{1}, \ldots, r_{n}\rangle$ be a finite presentation with all relators freely reduced, cyclically reduced, and distinct. For any $k \in \mathbb{N}$, define the finite presentation $P^{k}_{t}:=\langle x_{1}, \ldots, x_{m}, t \ | \ (r_{1}t)^{k}, \ldots, (r_{n}t)^{k}, t^{k}\rangle$. Then $P^{k}_{t}$ presents a $C'(2/k)$ small cancellation group. Moreover, for $k \geq 12$, we have $P_{t}^{k}/\tor_{1}(P_{t}^{k})  \cong P $ and so $\torlen(P_{t}^{k})=\torlen(P)+1$.
\end{thmn}

\vspace{3mm}

A part of the above theorem appeared in the work \cite{BumWis} of Bumagin and Wise; see Remark \ref{Bum remark}.

In Section \ref{quotients} we finish with a discussion of some open problems relating to torsion length and torsion subgroups.

\subsection{Notation}
A presentation $P=\langle X|R\rangle$ is said to be a \emph{recursive presentation} if $X$ is a finite set and $R$ is a recursive enumeration of relations. A group $G$ is said to be \emph{finitely} (respectively,   \emph{recursively}) \emph{presentable} if it can be presented by a  finite (respectively,   recursive) presentation. If $P,Q$ are group presentations denote their free product presentation by $P*Q$: this is given by taking the disjoint union of their generators and relations. If $g_{1}, \ldots, g_{n}$ are elements of a group $G$, we write $\langle   g_{1}, \ldots, g_{n} \rangle$ for the subgroup in $G$ generated by these elements and $\llangle g_{1}, \ldots, g_{n} \rrangle^{G}$ for the normal closure of these elements in $G$. Let $\omega$ denote the smallest infinite ordinal. Let $|X|$ denote the cardinality of a set $X$. If $X$ is a set, let $X^{-1}$ be a set of the same cardinality as and disjoint from $X$ along with a fixed bijection ${*}^{-1}: X \to X^{-1}$. Write $X^{*}$ for the set of finite reduced words on $X \cup X^{-1}$.

\subsection{Acknowledgements}
We thank Henry Wilton, Mark Hagen, Ben Barrett and Alan Logan for their assistance in preparing Section \ref{hyp1B}, and Andrew Glass and Jack Button for general comments and suggestions.

\section{$\tor_{n}(G)$, HNN extensions and Britton's lemma}

 \begin{defn} {\cite[Definition 3.1]{Chiodo3}} \label{defoftorsion}
Given a group $G$, inductively define $\tor_{n}(G)$ as follows:
\[
\tor_{0}(G):=\{e\},
\]
\[
\tor_{n+1}(G):=\llangle\ \{g \in G\ |\ g\tor_{n}(G) \in \tor \big( G/\tor_{n}(G)\big) \}\ \rrangle ^{G},
\]
\[
\tor_{\omega}(G):=\bigcup_{n \in \mathbb{N}}\tor_{n}(G).
\]
\end{defn}

Observe that $\tor_{i+1}(G)/\tor_{i}(G)=\tor_{1} \big( G/\tor_{i}(G) \big)$.

\begin{lem}{\cite[Corollary 3.4]{Chiodo3}}\label{leftadj}
$G/\tor_{\omega}(G)$ is torsion-free. Moreover, if $f:G\rightarrow H$ is a group homomorphism from $G$ to a torsion-free group $H$, then $\tor_{\omega}(G)\leq \ker(f)$, and so $f$ factors through $G/\tor_{\omega}(G)$.
\end{lem}

\begin{defn}{\cite[Definition 2.5]{ChiVya}} \label{tordefn}
The \emph{Torsion Length} of $G$, $\torlen(G)$, is the smallest ordinal $n$ such that $\tor_{n}(G)=\tor_{\omega}(G)$. 
\end{defn}

HNN extensions play a critical role in this paper; we briefly introduce them here.

\begin{defn}
Let $G$ be a group, and suppose there are isomorphisms $\varphi_{i}:A_{i}\to B_{i}$ for $1\leq i\leq n$, where $A_{i}$ and $B_{i}$ are subgroups of $G$. Define the \emph{HNN extension} $G*_{\varphi_{1}, \ldots, \varphi_{n}}$ with stable letters $t_{1},\ldots,t_{n}$ by
\[
 G*_{\varphi_{1}, \ldots, \varphi_{n}}:= (G*F_{n})/\llangle \{t_{i}^{-1}at_{i}\varphi_{i}(a^{-1}) \ | \ a \in A_{i}, 1\leq i\leq n\} \rrangle^{G*F_{n}},
\]
where $\{t_{1}, \ldots, t_{n}\}$ is a free generating set of $F_{n}$.

If $\varphi_{i}=\operatorname{id}_{A_{i}}$ for all $1\leq i\leq n$, we write $G*_{A_{1},\ldots, A_{n}}$ for $G*_{\varphi_{1}, \ldots, \varphi_{n}}.$
\end{defn}

\begin{defn}
Let $G*_{\varphi_{1}, \ldots, \varphi_{n}}$ be an HNN extension with stable letters $t_{1}, \ldots, t_{n}$. Then a $t_{i}$-\emph{pinch} is a word of the form  $t_{i}^{-1}gt_{i}$ where $g \in A_{i}$ or $t_{i}gt_{i}^{-1}$ where $g \in B_{i}$.
\\A word $w$ is said to be \emph{reduced} if no subword of $w$ is a $t_{i}$-pinch for any $i$. 
\end{defn}

\begin{thm}[Britton's lemma, {\cite[Theorem 11.81]{Rot}}]\label{Brit lem 2}
Let $H=G*_{\varphi_{1}, \ldots, \varphi_{n}}$ be an HNN extension with stable letters $t_{1}, \ldots, t_{n}$, and let $w$ be a word in $H$. If $w=e$ in $H$, then  $w$ contains a $t_{i}$-pinch as a subword, for some $i$.
\end{thm}

\begin{cor}\label{HNN emb2}
Let $G*_{\varphi_{1}, \ldots, \varphi_{n}}$ be an HNN extension. Then $G$ embeds into $G*_{\varphi_{1}, \ldots, \varphi_{n}}$.
\end{cor}

Given a group $G$ we write $\langle G; X|R \rangle$ to denote $(G*F_{X})/\llangle R \rrangle^{G * F_X}$, where $R$ is any subset of $G*F_{X}$.

\section{Good subgroups of HNN extensions}

The notion of a good subgroup was introduced in \cite[Proposition 1.34]{Cohen}, and named so in \cite[Definition 2]{Simpson}.

\begin{defn}
Let $H=G*_{\varphi_{1}, \ldots, \varphi_{n}}$ be an HNN extension. A \emph{good subgroup} of $G$ with respect to the HNN extension $H$ is a subgroup $K \leq G\leq H$ such that $\varphi_{i}(K\cap A_{i}) =K\cap B_{i}$ for all $1 \leq i \leq n$.
\end{defn}

\begin{lem}{{\cite[Proposition 1.34]{Cohen}}}
Let $H:= G*_{\varphi_{1}, \ldots, \varphi_{n}}$ be an HNN extension of $G$ with stable letters $t_{1}, \ldots, t_{n}$.  Suppose $K \leq G$ is a good subgroup of $G$ with respect to the HNN extension $H$, and let $\psi_{i}: K\cap A_{i}\to K\cap B_{i}$ be the restriction of $\varphi_{i}$ to $K\cap A_{i}.$  Let $K'$ be the subgroup of $H$ generated by $K$ and $ t_{1}, \ldots, t_{n}$. Then, the natural map \[\nu_{K}: K*_{\psi_{1}, \ldots, \psi_{n}}\to K'\] is an isomorphism. Moreover, $K' \cap G=K$.
\end{lem}

We now study good subgroups which are normal.

\begin{defn}\label{good pieces}
  Let $H:= G*_{\varphi_{1}, \ldots, \varphi_{n}}$ be an HNN extension of $G$ with stable letters $t_{1}, \ldots, t_{n}$. Let $K \nsub G$ be a good subgroup of $G$ with respect to the HNN extension $H$.  Let $\overline{\varphi}_{i}:A_{i}/(K \cap A_{i}) \to B_{i}/(K \cap B_{i})$ be the induced isomorphism for each $1 \leq i \leq n$. Define the following HNN extension with stable letters $\overline{t}_{1}, \ldots, \overline{t}_{n}$: \[H_{K}:= (G/K)*_{\overline{\varphi}_{1}, \ldots, \overline{\varphi}_{n}}\]
  There is a surjective homomorphism  \[\phi_{K}: H \twoheadrightarrow H_{K}\] which sends $g\mapsto gK$ for all $g \in G$, and $t_{i} \mapsto \overline{t}_{i}$  for all $1 \leq i \leq n$.
\end{defn}

\begin{lem}\label{good normal lem}
 Let $G$ be a group, and $H:=G*_{\varphi_{1}, \ldots, \varphi_{n}}$ an HNN extension of $G$ with stable letters $t_{1}, \ldots, t_{n}$.  Let $K \nsub G$. Then $K$ is a good subgroup of $G$  with respect to the HNN extension $H$ if and only if  $\llangle K \rrangle^{H} \cap G = K$ in  $H.$

\end{lem}
 
\begin{proof}\mbox{}
\\$\Leftarrow$: Assume that $\llangle K \rrangle^{H} \cap G = K$ in $H$. Take  $1\leq i \leq n$ and suppose $x \in A_{i} \cap K$. We know that $\varphi_{i}(x)\in B_{i}$: we need to verify that  $\varphi_{i}(x)\in K$. However,  it is immediate that $\varphi_{i}(x)=t_{i}^{-1}xt_{i}\in \llangle K \rrangle^{H}$, and thus that $\varphi_{i}(x)\in \llangle K \rrangle^{H}\cap B_{i}=\llangle K \rrangle^{H}\cap G \cap B_{i} = K\cap B_{i};$ it follows that $\varphi_{i}(A_{i}\cap K) \subseteq B_{i} \cap K.$ The inclusion $\varphi_{i}(A_{i}\cap K) \supseteq B_{i} \cap K$ can be proved in a similar fashion. 
\\$\Rightarrow$: Suppose $K$ is a good subgroup of  $G$  with respect to the HNN extension $H$, and take $\phi_{K}$ as in Definition \ref{good pieces}.
Then it is clear that $K \leq \llangle K \rrangle^{H}\cap G \leq \kernel(\phi_{K})\cap G \leq K$; the last inequality here is a consequence of Theorem \ref{Brit lem 2}.
\end{proof}

\begin{lem}\label{good kernel}
Let $H$, $K$ and $\phi_{K}$ be as in Definition $\ref{good pieces}$. Then $\kernel(\phi_{K}) = \llangle K \rrangle ^{H}$.
\end{lem}

\begin{proof}

The containment $ \llangle K \rrangle ^{H} \subseteq \kernel(\phi_{K}) $ is immediate.

Let $x \in \kernel(\phi_{K})$. We induct on the total number of occurrences of $t_{i}$ or $t_{i}^{-1}$ over all the $i$'s, where $1\leq i\leq n$,  in the normal form of $x$ in $H$: if $x$ has none, then $x\in K.$ 
 
Assume that for some $i$, either $t_{i}$ or $t_{i}^{-1}$ appears at least once in $x$.  By Britton's lemma $\phi_{K}(x)$ has a subword of the form $\overline{t}_{i}^{-1}a\overline{t}_{i}$ where $ a \in A_{i}/(A_{i} \cap K)$ or $\overline{t}_{i}b\overline{t}_{i}^{-1}$ where $b \in B_{i}/(B\cap K)$. Thus $x$ has a subword of the form $t_{i}^{-1}a't_{i}$ where $ a' \in A_{i}K$ or $t_{i}b't^{-1}$  where $b \in B_{i}K$. Without loss of generality, we assume the former.  
 
This subword $t_{i}^{-1}a't$ is of the form $ t_{i}^{-1}akt_{i}$, where $a \in A_{i}$ and  $k \in K.$ 
But $t_{i}^{-1}at_{i}=b \in B$, for some $b\in B$.  We can therefore write $x$ as $\lambda_{1}t_{i}^{-1}akt_{i}\lambda_{2}=\lambda_{1}bt_{i}^{-1}kt_{i}\lambda_{2}$.  Observe that $t_{i}^{-1}kt_{i}\in \llangle K \rrangle^{H}$, and thus that $t_{i}^{-1}kt_{i}\lambda_{2} = \lambda_{2}y$ where  $y \in \llangle K \rrangle^{H}$. We can therefore rewrite $x=\lambda_{1}b\lambda_{2}y$; from this we see that $\lambda_{1}b\lambda_{2} \in \kernel(\phi_{K})$. By induction, we have that $\lambda_{1}b\lambda_{2} \in \llangle K \rrangle^{H}$. This tells us that $x\in \llangle K \rrangle^{H}$.
\end{proof}

\begin{cor}\label{subset lem cor}
 Let $H$, $K$ and $H_{K}$ be as in Definition $\ref{good pieces}$. Then $\phi_{K}$ induces an isomorphism 
 \[
 \overline{\phi}_{K}\ :\  H/\llangle K \rrangle^{H} \overset{\cong}{\longrightarrow} H_{K}.
 \]
\end{cor}

\section{The Higman embedding construction}\label{constructions}

The Higman Embedding Theorem states that a finitely generated, recursively presented group can be embedded in a finitely presented group. In this section we provide an overview of a proof of this result, introducing notation and constructions that will be used later in this note.

\subsection{Simulating a modular machine by a finitely presented group}\label{fg group}\mbox{}

We begin by describing how a modular machine can be simulated by a finitely presented group. This construction is then used in a proof of the Higman Embedding Theorem. 

The idea is to follow the construction in \cite[pp.266--268]{Cohen}. This was derived from \cite{Cohen2}, where a detailed exposition of modular machines can be found. We felt, however, that the exposition in \cite{Simpson} was slightly clearer, so we replicate here the argument presented there (the differences are only slight).

\begin{enumerate}
 \item Define the group $K:=\langle x,y,t \ | \ [x,y]=e \rangle$.
 \item For all $(r,s)\in \mathbb{Z}^{2}$, define the word $t(r,s):= y^{-s}x^{-r}tx^{r}y^{s}\in K.$
 \item \label{s:1} Let $T := \langle \{t(r,s)\}_{(r,s)\in \mathbb{Z}^{2}} \rangle \leq K$.
  \item Observe that $T$ is free with basis $\{t(r,s) \}_{(r,s)\in \mathbb{Z}^{2}}$.
 \item Observe that $T=\llangle t \rrangle^{K}$.
 
 \item For $M>a\geq 0$, $N>b\geq 0$, define
 \[  K_{a,b}^{M,N}:=\langle  t(a,b),x^{M}, y^{N} \rangle \leq K,\]
 \[
  T_{a,b}^{M,N}:=\langle \{t(\alpha, \beta) \ | \ \alpha \equiv a \mod M,\ \beta \equiv b \mod N  \} \rangle  \leq T \leq K.
 \]
 
 \item Let $(i,j)\in \mathbb{Z}^{2},$ and $m,n\in \mathbb{Z}.$ Observe that $T\cap \langle t(i,j), x^{m}, y^{n}\rangle$ is free with basis $\{t(r,s) \ | \ r \equiv i \mod m,\ s \equiv j \mod n\}$. In particular, \[T \cap K_{a,b}^{M,N} = T_{a,b}^{M,N}.\]
 
 \item\label{map step} Observe that the correspondence $t\mapsto t(a,b)$, $x\mapsto x^{M}$, $y\mapsto y^{N}$ induces an isomorphism \[K\to K_{a,b}^{M,N}.\] This isomorphism sends $t(u,v)$ to $ t(uM+a, vN+b)$ and thus induces an isomorphism \[T\to T_{a,b}^{M,N}.\]

 \item Let $\mathcal{M}=\{(a_{i},b_{i}, c_{i}, R )\ | \ i \in I \}\cup \{(a_{j},b_{j}, c_{j}, L )\ | \ j \in J \}$ be a modular machine with modulus $m$.
 \item The maps in step (\ref{map step}) induce, for each $i\in I$ and $j\in J$, isomorphisms \[\phi_{i}: K_{a_{i}, b_{i}}^{m,m} \to K_{c_{i}, 0}^{m^{2},1},\]  \[\varphi_{j}: K_{a_{j}, b_{j}}^{m,m} \to K_{0, c_{j}}^{1,m^{2}}.\] 
 \item\label{define KM} Form the HNN extension \[K_{\mathcal{M}}:= K*_{\{\phi_{i}\}_{i\in I}, \{\varphi_{j}\}_{j\in J}},\] with stable letters $\{r_{i}\}_{i\in I}$ and $\{l_{j}\}_{j\in J}$.
Note that $K_{\mathcal{M}}$ is finitely presented.
 \item Define the subgroup $T':=\langle T, \{r_{i}\}_{i \in I}, \{l_{j}\}_{j \in J} \rangle \leq  K_{\mathcal{M}}$, where $T$ is as in step (\ref{s:1}).
 \item Define the set $H_{0}(\mathcal{M}):= \{ (\alpha, \beta) \ | \ (\alpha, \beta) \underset{\mathcal{M}}{\overset{*}{\longrightarrow}}  (0,0) \}.$
 
 \item Define $T_{\mathcal{M}}:=\langle \{t(\alpha, \beta) \ | \ (\alpha, \beta) \in H_{0}(\mathcal{M})\} \rangle\leq K.$
 
 \item Define $T_{\mathcal{M}}':=\langle T_{\mathcal{M}}, \{r_{i}\}_{i \in I}, \{l_{j}\}_{j \in J} \rangle \leq K_{\mathcal{M}}.$

 \item Observe that $T_{\mathcal{M}}'=\langle t,\{ r_{i}\}_{i \in I}, \{l_{j}\}_{j \in J} \rangle.$ 
 
 \item Observe that $t(\alpha, \beta) \in T_{\mathcal{M}}'$ iff $(\alpha, \beta) \in H_{0}(\mathcal{M})$.

 \item With the identity map $\theta: T_{\mathcal{M}}' \to T_{\mathcal{M}}'$, form the HNN extension \[G_{\mathcal{M}}:= K_{\mathcal{M}}*_{\theta}\] with stable letter $q$. 

 \item\label{key point} Observe that $q^{-1}t(\alpha, \beta)q=t(\alpha, \beta)$ in $G_{\mathcal{M}}$ iff $(\alpha, \beta) \in H_{0}(\mathcal{M})$.
  
\end{enumerate}

Taking $\mathcal{M'}$ with nonrecursive halting set $H_{0}(\mathcal{M'})$ gives a finitely presented group $G_{\mathcal{M'}}$ with undecidable word problem. 

For our purposes, a useful consequence of the above construction is that we can simulate any modular machine by a finitely generated group: see step (\ref{key point}) of Construction \ref{fg group}.

\subsection{The Higman Embedding Theorem}\label{hig construct}\mbox{}

We now give an overview of the construction used in a particular proof of the Higman Embedding Theorem, taken directly from \cite[pp.279--281]{Cohen}. We note that this proof originally comes from \cite{Cohen3}.

\begin{enumerate}
 \item Let $C=\langle c_{1}, \ldots, c_{n} \ | \ S\  \rangle$ be a finitely generated recursively presented group, where $S$ corresponds to the set $H_{0}(\mathcal{M})$ of a modular machine $\mathcal{M}$; see step (\ref{S set}). Denote the modulus of $\mathcal{M}$ by $m$.  We assume that  $S$ covers \emph{all} the trivial words in the group. 
 \item Re-write every word in $C$ as a word in the free monoid on $\{c_{1}, \ldots, c_{2n}\}$ with $c_{i}^{-1}$ replaced by $c_{n+i}$.
 \item To each word $w=c_{i_{k}}c_{i_{k-1}}\cdots c_{i_{0}}$ associate an $m$-ary representation $\alpha=\Sigma_{j=0}^{k}i_{j}m^{j}$. 
 \item Define $I:=\{\alpha \in \mathbb{N} \ | \ \alpha\ \textnormal{represents a word} \}$. That is, $\alpha=\Sigma_{j=0}^{k}\beta_{j}m^{j}$ where  $1 \leq \beta_{j} \leq 2n.$
 \item For $\alpha \in I$, define $w_{\alpha}(c)$  to be the word formed from $\alpha$.
 \item For $\alpha \in I$, write  $w_{\alpha}(b)$, $w_{\alpha}(bc)$ for the words obtained from $w_{\alpha}(c)$ by replacing $c_{i}$ with $b_{i}$ and $b_{i}c_{i}$ respectively (where $\{b_{i}\}_{i=1}^{2n}$ are a new set of symbols).
 \item\label{S set} Observe that, for all $\alpha \in I$, we have that $w_{\alpha}(c) \in S$ iff $(\alpha, 0) \in H_{0}(M)$.
 \item Recall the group $K_{\mathcal{M}}$ from step (\ref{define KM}) of \ref{fg group}. Define $U:=\{t, \{r_{i}\}_{i\in I}, \{l_{j}\}_{ j\in J} \}$; $U$ is a subset of $K_{\mathcal{M}}.$ 
 \item Define $t_{\alpha}:=t(\alpha, 0)\in K_{\mathcal{M}}.$
 \item Form the free product
 \[
 H_{1}:=K_{\mathcal{M}}* (C \times \langle b_{1}, \ldots, b_{n} \vert - \rangle)*\langle d  \vert -\rangle,
 \]
 and set $b_{n+i}:=b_{i}^{-1}$ for $1 \leq i \leq n$.
 \item Observe that the sets $\{t_{\alpha}\ |\  \alpha \in I\}$ and $\{t_{\alpha}w_{\alpha}(b)d\ |\  \alpha \in I\}$ each form a free basis for the subgroups they respectively generate in $H_{1}$. The correspondence $t_{\alpha} \mapsto t_{\alpha}w_{\alpha}(b)d $ extends to an isomorphism $\psi$ between these subgroups. 
 \item Form the HNN extension
 \[
  H_{2}:= H_{1}*_{\psi}
 \] with stable letter $p$. 
 \item Define the subgroup
 \[
  A:=\langle t,x,d,b_{1}, \ldots, b_{n}, p \rangle \leq H_{2}.
 \]
 \item For $1 \leq i \leq 2n$, define the subgroup
 \[
  A_{i}:=\langle t_{i},x^{m},b_{i}d,b_{1}, \ldots, b_{n}, p \rangle \leq H_{2}.
 \]
 \item \label{s:5} Observe that for all $i$, $A$ is isomorphic to $A_{i}$ via the map $\psi_{i}$ induced by the correspondence sending $t \mapsto t_{i}$, $x \mapsto x^{m}$, $d \mapsto b_{i}d$, $b_{j} \mapsto b_{j}$ for all $\ 1 \leq j \leq n$, and $p \mapsto p$.
 
 \item Observe that $\langle t,x,d,b_{1},\ldots,b_{n}\rangle$ and $\langle t_{i},x^{m},b_{i}d,b_{1}, \ldots, b_{n} \rangle$ are both good in $H_{1}$ with respect to the HNN extension $H_{2}$. Therefore $A$, and the $A_{i}$ for $1 \leq i \leq 2n$, are all HNN extensions.
 \item Define the subgroup
 \[
  A_{+}:=\langle U,d,b_{1}, \ldots, b_{n}, p \rangle \leq H_{2}. 
 \]
\item Define the subgroup
 \[
  A_{-}:=\langle U,d,b_{1}c_{1},\ldots, b_{n}c_{n}, p \rangle \leq H_{2}. 
 \]

  \item Observe that $\langle U,d,b_{1}, \ldots, b_{n} \rangle $ is good in $H_{1}$ with respect to the HNN extension $H_{2}$. Therefore $A_{+}$ is an HNN extension.
 
 \item\label{iso +-} Observe that $A_{+}$ is isomorphic to $A_{-}$ via the map $\psi_{+}: A_{+}\to A_{-}$  induced by the correspondence sending $u\mapsto u$ for all $u \in U$, $d \mapsto d$, $b_{j} \mapsto b_{j}c_{j}$ for all  $1 \leq j \leq n$, and $p \mapsto p$.

 \item With the isomorphisms defined above, define the HNN extension
 \[ H_{3}:= H_{2}*_{\psi_{1},\ldots,\psi_{2n},\psi_{+}},\]
 with stable letters $a_{1},\ldots,a_{2n}$ and $k$.

 \item Observe that $H_{3}$ is finitely presented, and $C \hookrightarrow H_{3}$.

\end{enumerate}

\section{Properties of the embedding construction}

In this section, the groups $C, H_{1}, H_{2}$ and $H_{3}$ will be as in Section \ref{hig construct}.

\begin{lem}\label{subset lem}
 Let $X$ be a subset of $ C$. Then
 
 \begin{enumerate}
\item $\llangle X \rrangle^{H_{1}}$ is good in $H_{1}$  with respect to the HNN extension $H_{2}.$
\item $\llangle X \rrangle^{H_{2}}$ is good in $H_{2}$  with respect to the HNN extension $H_{3}.$
 \end{enumerate}

\end{lem}

\begin{proof}
 We claim that the following is true:
 \begin{align*}
  \llangle X \rrangle^{H_{1}} &\cap \langle \{t_{\alpha} \ | \ \alpha \in I\} \rangle=\{e\},
  \\ \llangle X \rrangle^{H_{1}} &\cap \langle \{t_{\alpha}w_{\alpha}(b)d \ | \ \alpha \in I\} \rangle =\{e\}.
 \end{align*}
To see this, consider the map $ \lambda :  H_{1} \to H_{1}$ induced by the identity maps on $K_{\mathcal{M}}$, $\langle b_{1}, \ldots, b_{n} \vert - \rangle$, and  $\langle d \vert - \rangle$, and the trivial map on $C$.

The map $\lambda$,  restricted to ${K_{M}* (\{e\} \times \langle b_{1}, \ldots, b_{n} \vert -  \rangle)*\langle d  | -\rangle}$, is injective, and thus injective on both  $\langle \{t_{\alpha} \ | \ \alpha \in I\} \rangle$ and  $\langle \{t_{\alpha}w_{\alpha}(b)d \ | \ \alpha \in I\} \rangle $. However, $\llangle X \rrangle^{H_{1}}$ is contained in $ \kernel(\lambda)$.  

This proves the first part of the lemma; we now move to the second. 

Take the map $\lambda$ defined above. It is clear that $\lambda$ extends to a map $\overline{\lambda}: H_{2} \to H_{2}$, sending $p \mapsto p$. 
Again, $\llangle X \rrangle^{H_{2}} \leq \kernel(\overline{\lambda})$. As before, we see that the restriction of $\overline{\lambda}$ to $ {\langle K_{M}* (\{e\} \times \langle b_{1}, \ldots, b_{n} | - \rangle)*\langle d  | -\rangle , p \rangle}$ is injective. It follows that $\llangle X \rrangle^{H_{2}}\cap A = \llangle X \rrangle^{H_{2}}\cap A_{i}=\{e\}$ for all $1 \leq i \leq 2n$.
\\Finally, consider the inclusions
\begin{align*}
 &\iota_{-}: A_{-}\to H_{2}
 \\ &\iota_{+}: A_{+}\to H_{2}
\end{align*}
Step (\ref{iso +-}) of \ref{hig construct} tells us that the restriction of $\lambda$ to $A_{-}$ is injective with image $A_{+}$, and thus induces an isomorphism $\lambda': A_{-} \to A_{+}$; $\lambda'$ is inverse to the map $\psi_{+}$ defined in (\ref{s:5}) of \ref{hig construct}. We see that $\overline{\lambda}\circ \iota_{+} \circ \lambda'=\overline{\lambda}\circ\iota_{-}: A_{-}\to H_{2}$:

\begin{displaymath}
\xymatrix{
A_{-} \ar[d]_{\lambda'} \ar[rd]^{\overline{\lambda}\circ i_{-}} &  \\
A_{+} \ar[r]_{\overline{\lambda}\circ i_{+}} &  H_{2}}
 \end{displaymath}
It is clear that $\overline{\lambda}\circ\iota_{+}$ is injective, and thus that $\overline{\lambda}\circ\iota_{-}$ is as well. Since $\llangle X \rrangle^{H_{2}}$ is contained in $\kernel(\overline{\lambda})$, we see that $\llangle X \rrangle^{H_{2}}\cap A_{-} = \llangle X \rrangle^{H_{2}}\cap A_{+}=\{e\}$. This proves the last part of the lemma.
\end{proof}

Before we proceed, we need the following observation. It is proved in the same way that \cite[Corollary 2.9]{ChiVya} is, by using the torsion theorem for HNN extensions.

\begin{lem}\label{HNN tor lem}
 Let $G$ be a group, and $\varphi: H \to K$ an isomorphism between subgroups $H, K \leq G$. Let $G*_{\varphi}$ be the associated HNN extension. Then 
 \[
  \tor_{1}(G*_{\varphi})=\llangle \tor_{1}(G) \rrangle^{G*_{\varphi}} = \llangle \tor(G) \rrangle^{G*_{\varphi}}.
  \]
\end{lem}

\begin{lem}\label{C results}
For all $m\geq 0$, the following hold:
\begin{enumerate}
\item $\tor_{m}(H_{1})=\llangle \tor_{m}(C) \rrangle^{H_{1}}$.
\item $\tor_{m}(H_{1})\cap C=\tor_{m}(C)$.
\end{enumerate}

\end{lem}

\begin{proof}\mbox{}

By \cite[Proposition 2.10]{ChiVya}, we know that \[\tor_{m}(H_{1})=\llangle \tor_{m}(K_{\mathcal{M}}) \cup \tor_{m}(C \times \langle b_{1}, \ldots, b_{n}\vert - \rangle)\cup \tor_{m}(\langle d|-\rangle) \rrangle^{H_{1}}.\] However, $K_{\mathcal{M}}$, $\langle b_{1}, \ldots, b_{n} \vert - \rangle$ and $\langle d|-\rangle$ are all torsion-free. It follows that \[\tor_{m}(H_{1})=\llangle \tor_{m}(C) \rrangle^{H_{1}}\] for all $m$. This proves part (1). 
For the second part, observe that there is a map $\mu: H_{1}\to C$ induced by the trivial map on $K_{\mathcal{M}}$ and $\langle d| -\rangle$ and the standard projection to $C$ on $C \times \langle b_{1}, \ldots, b_{n} \vert - \rangle.$ The map $\mu$ restricts to the identity on $C$ and sends $\tor_{m}(H_{1})$ to $\tor_{m}(C).$ The result follows. 
\end{proof}

\begin{lem}\label{H results}
For $i= 1,2$, and for all $m\geq 0$, the following hold: 
 \begin{enumerate}
  \item $\tor_{m}(H_{i+1})=\llangle \tor_{m}(H_{i}) \rrangle^{H_{i+1}}$.
 \item $\tor_{m}(H_{i})$ is good in $H_{i}$  with respect to the HNN extension $H_{i+1}$.
 \end{enumerate}
\end{lem}

\begin{proof}\mbox{}
We prove this by induction on $m$.  The result is obvious for $m=0$. 

We now come to the inductive step. Let $i\in \{1,2\}$.  Assume the statement is true for $m$. The induction hypothesis tells us that  $\tor_{m}(H_{i+1})=\llangle \tor_{m}(H_{i}) \rrangle^{H_{i+1}}$ and that $\tor_{m}(H_{i})$ is good in $H_{i}$  with respect to the HNN extension $H_{i+1}$. Thus, by Lemma \ref{good normal lem}, \[ \llangle \tor_{m}(H_{i})  \rrangle^{H_{i+1}} \cap H_{i}= \tor_{m}(H_{i}).\] Combining these facts, we see that $ \tor_{m}(H_{i+1}) \cap H_{i} =\tor_{m}(H_{i})$. As a consequence, the inclusion $H_{i}\to H_{i+1}$ induces an embedding \[H_{i}/\tor_{m}(H_{i})\to H_{i+1}/\tor_{m}(H_{i+1});\] via this, we identify $H_{i}/\tor_{m}(H_{i})$ as a subgroup of $H_{i+1}/\tor_{m}(H_{i+1})$.

Using Lemma \ref{good kernel} and Corollary \ref{subset lem cor}, we see that
\[
 H_{i+1}/\tor_{m}(H_{i+1}) = H_{i+1}/\llangle \tor_{m}(H_{i})\rrangle^{H_{i+1}} \cong \langle H_{i}/\tor_{m}(H_{i}); \stable_{i}\  | \ \relations_{i} \rangle
\]
where $\stable_{i}$, $\relations_{i}$ are, respectively, the stable letters and relations of the HNN construction of $H_{i+1}$ from $H_{i}$. It then follows from Lemma \ref{HNN tor lem} that 
\[
 \tor_{1}(H_{i+1}/\tor_{m}(H_{i+1}))  = \llangle  \tor_{1}(H_{i}/\tor_{m}(H_{i})) \rrangle^{H_{i+1}/\tor_{m}(H_{i+1})}.
\]
The preimage of $\tor_{1}(H_{i+1}/\tor_{m}(H_{i+1}))$ in $H_{i+1}$ is  $\tor_{m+1}(H_{i+1})$, and the preimage of $\llangle  \tor_{1}(H_{i}/\tor_{m}(H_{i}) \rrangle^{H_{i+1}/\tor_{m}(H_{i+1})}$ in $H_{i+1}$ is  $\llangle \tor_{m+1}(H_{i}) \rrangle^{H_{i+1}}$. Thus $\tor_{m+1}(H_{i+1})=\llangle \tor_{m+1}(H_{i}) \rrangle^{H_{i+1}}$, and (1) is proved for the case $m+1$.

We have just proved that (1) is true for $m+1$; combining this fact with Lemma \ref{C results} (1),  we see that  \[ \tor_{m+1}(H_{i})=\llangle \tor_{m+1}(H_{i-1})\rrangle^{H_{i}}= \ldots = \llangle \tor_{m+1}(H_{1})\rrangle^{H_{i}}=\llangle \tor_{m+1}(C)\rrangle^{H_{i}}.\] Lemma \ref{subset lem} then tells us that $\tor_{m+1}(H_{i})$ is good in $H_{i}$  with respect to the HNN extension $H_{i+1}$.
\end{proof}

The next corollary now follows from Lemmas \ref{C results}, \ref{H results} and \ref{good normal lem}:

\begin{cor}\label{C cor}
 For $i=1,2,3$, and for all $m\geq 0$, the following hold:
 \begin{enumerate}
 \item $\tor_{m}(H_{i})=\llangle \tor_{m}(C)\rrangle^{H_{i}}.$
 \item $\tor_{m}(H_{i})\cap C=\tor_{m}(C)$.
 \end{enumerate}
\end{cor}

\begin{thm}\label{preserve torlen}
There is a uniform construction that, on input of a recursive presentation of a group $C$, outputs a finite presentation of a group $H$ in which $C$ embeds, with $\torlen(C)=\torlen(H)$.
\end{thm}

\begin{proof}
This is an immediate consequence of Corollary \ref{C cor}, taking $H=H_{3}$. As \[\llangle \tor_{m}(C) \rrangle^{H_{3}}  =\tor_{m}(H_{3})\] for all $m$, $\tor_{m}(C)= \tor_{m+1}(C)$  implies that  $\tor_{m}(H_{3})= \tor_{m+1}(H_{3})$.

Conversely, since \[\tor_{m}(H_{3})\cap C=\tor_{m}(C)\] for all $m$, $\tor_{m}(H_{3})= \tor_{m+1}(H_{3})$ implies that $\tor_{m}(C)= \tor_{m+1}(C)$.

In conclusion, $\tor_{m}(H_{3})= \tor_{m+1}(H_{3})$ if and only if $\tor_{m}(C)=\tor_{m+1}(C)$, for any $m$. Thus the sequences $\tor_{j}(H_{3})$ and $\tor_{j}(C)$ stabilise at precisely the same value of $j$ (if at all), and so $\torlen(H_{3})= \torlen(C)$.
\end{proof}

\begin{thm} \label{txfpg}
There exists a finitely presented group $F$ with $\torlen(F) = \omega$.

\end{thm}
\begin{proof}
In \cite[Theorem 3.10]{ChiVya}, we proved that there is a $2$-generator, recursively presented group with infinite torsion length. We now apply Theorem \ref{preserve torlen}.
\end{proof}

An interesting exercise would be to construct an explicit finite presentation of such a group. Theoretically, this could  be done by carefully following the constructions given above. The presentation that arises as the output of such a process, however, would undoubtedly be very complicated.  A more straightforward presentation, perhaps giving a group that arises elsewhere in the literature, would be interesting.

\section{A word-hyperbolic virtually special example}\label{hyp1B}

We now show various ways of constructing finitely presented virtually special groups with infinite torsion length. We thank Henry Wilton for initially suggesting that this is possible and pointing out an alternate way to prove it.

\begin{defn}
Let $\Gamma$ be an undirected graph on finite vertex set labeled $1, \ldots, n$, and edge set $E$. The \emph{right-angled Artin group} (RAAG), $A(\Gamma)$, associated to $\Gamma$ is the group with presentation \[\langle x_{1}, \ldots, x_{n} \ | \ [x_{i}, x_{j}]\  \forall \ \{i,j\} \in E\rangle.\] 

A group $G$ is said to be \emph{special} if it is a subgroup of some RAAG. More generally, a group $G$ is said to be \emph{virtually special} if it contains a finite index subgroup which is special.
\end{defn}

Every RAAG on $n$ generators can be seen as an HNN extension of a RAAG on $n-1$ generators; it follows that every RAAG is torsion-free, and thus that virtually special groups are virtually torsion-free.

For the remainder of this section, if $P=\langle X|R \rangle$ is a group presentation we denote by $\overline{P}$ the group presented by $P$, and if $w\in X^{*}$ is a word in the generators of $P$ then we denote by $\overline{w}$ the element of $\overline{P}$ represented by $w$.

\begin{defn}
 Let $P=\langle X|R\rangle$ be a finite presentation where each $r\in R$ is freely reduced and cyclically reduced (as a word in $X^{*}$), and where $R$ is \emph{symmetrised} (i.e.,  closed under taking cyclic permutations and inverses).
 \\A nonempty freely reduced word $w\in X^{*}$ is called a \emph{piece} with respect to $P$ if there exist two distinct elements $r_{1}, r_{2} \in R$ that have $w$ as maximal common initial segment.
 \\Let $0 < \lambda < 1$. Then $P$ is said to satisfy the $C'(\lambda)$ \emph{small cancellation condition} if whenever $w$ is a piece with respect to $P$ and $w$ is a subword of some $r \in R$, then $|u| < \lambda|r|$ as words.
 \\ A group is called a $C'(\lambda)$ group if it can be presented by a presentation satisfying the $C'(\lambda)$ \emph{small cancellation condition}.
\end{defn}

If  $P=\langle X|R\rangle$  is a finite presentation of a group $G$ where $R$ is not symmetrised, we can take the symmetrised closure $R_{sym}$ of $R$, where $R_{sym}$ consists of all cyclic permutations of words in $R$ and $R^{-1}$ (with repetitions removed). Then $R_{sym}$ is symmetrised and $P_{sym}=\langle X| R_{sym}\rangle$  is also a presentation of $G$. In a slight abuse of notation, we call the presentation $P=\langle X  |R \rangle$ symmetrised if $R$ is symmetrised.

The following theorem is a consequence of the substantial results of Agol \cite{Agol} and Wise \cite{Wise}.

\begin{thm}\label{agol wise}
 Let $P=\langle X|R \rangle$ be symmetrised and satisfy the $C'(1/6)$ small cancellation condition. Then $\overline{P}$ is both word-hyperbolic and virtually special.
\end{thm}

\begin{proof}
$C'(1/6)$ groups are known to be word-hyperbolic. One then uses \cite[Theorem 1.2]{Wise} and \cite[Theorem 1.1]{Agol} to show that $\overline{P}$ is virtually special.
\end{proof}

\begin{prop}\label{powersB}
 Let $P=\langle x_{1}, \ldots, x_{m} \ | \ r_{1}, \ldots, r_{n}\rangle$ be a finite presentation, with all words in $R$ freely reduced, cyclically reduced, and distinct. For any $k \in \mathbb{N}$, define the finite presentation
 \[
 P^{k}_{t}:=\langle x_{1}, \ldots, x_{m}, t \ | \ (r_{1}t)^{k}, \ldots, (r_{n}t)^{k}, t^{k}\rangle
 \]
Then  $(P^{k}_{t})_{sym}$ is symmetrised and satisfies the $C'(2/k)$ small cancellation condition, and thus  $\overline{P^{k}_{t}}$ is word-hyperbolic and virtually special for all  $k \geq 12$.
\end{prop}

\begin{proof}
Let $S=\{ (r_{1}t)^{k}, \ldots, (r_{n}t)^{k}, t^{k} \}$. We first need to check that every $s \in S_{sym}$  is freely and cyclically reduced. But this follows from the fact that $R$ is freely and cyclically reduced, along with the strategic placements of the $t$'s. By definition, $S_{sym}$ is symmetrised. We now show that $S_{sym}$ satisfies the  $C'(2/k)$ small cancellation condition. 
\\1. Any cyclic permutation of $(r_{i}t)^{k}$ shares a piece with $t^{k}$ of length at most one (and no piece with $t^{-k}$). Similarly, any cyclic permutation of $(r_{i}t)^{-k}$ shares a piece with $t^{-k}$ of length at most one (and no piece with $t^{k}$). Such pieces have length at most $1/k$ of either word.
\\2.  Consider shared pieces of cyclic permutations of pairs of words of the form $(r_{i}t)^{k}$ and $(r_{j}t)^{k}$. If $r_{i}=ab$ and $r_{j}=cd$, where $a,b,c,d$ are words, then we are left with considering words of the form $bt(r_{i}t)^{k-1}a$ and $dt(r_{j}t)^{k-1}c$ respectively. As $r_{i}\neq r_{j}$, the initial overlap of these can be at most $bt\equiv dt$, followed by some overlap of $r_{i} $ and $r_{j}$ (of length at most $\min\{|r_{i}|, |r_{j}|\}$, as $r_{i }\neq r_{j}$ and $t$ is acting as an end marker). So this initial overlap can have length at most $\min\{2|r_{i}|, 2|r_{j}|\}+1$ which is less than $2/k$ of the length of either word.
\\3. By repeating the arguments as in step 2, we can show that cyclic permutations of pairs of words of the form $(r_{i}t)^{-k}$ and $(r_{j}t)^{-k}$ overlap at most $2/k$ of the length of either word.
\\4. We now consider shared pieces of cyclic permutations of pairs of words of the form $(r_{i}t)^{k}$ and $(r_{j}t)^{-k}=(t^{-1}r_{j}^{-1})^{k}$. If $r_{i}=ab$ and $r_{j}=cd$, where $a,b,c,d$ are words, then the words we are considering must be of the form $bt(r_{i}t)^{k-1}a$ and $c^{-1}(t^{-1}r_{j}^{-1})^{(k-1)}t^{-1}d^{-1}$ respectively. An initial overlap cannot involve $t$ or $t^{-1}$, and thus has length at most $\min\{|r_{i}|, |r_{j}|\}$; this is less than $1/k$ of the length of either word.

If follows that  $(P^{k}_{t})_{sym}$ satisfies the $C'(2/k)$ small cancellation condition; we appeal to Theorem \ref{agol wise} to finish the proof.
\end{proof}

The following standard result was first proved in \cite{Gre}; see \cite[Theorem 6]{Sch} for an explicit statement of the result.

\begin{lem}[{\cite[Theorem VIII]{Gre}}]\label{C 1/6}
 Let $P=\langle X|R\rangle$ satisfy the $C'(1/6)$ small cancellation condition. Then an element $g \in \overline{P}$ has order $n > 1$ if and only if there is a relator $r \in R$ of the form $r = s^{n}$ in $X^{*}$, with $s\in X^{*}$, such that $g$ is conjugate to $\overline{s}$ in $\overline{P}$.
\end{lem}

\begin{lem}\label{clever torsion}
 Let $P=\langle x_{1}, \ldots, x_{m} \ | \ r_{1}, \ldots, r_{n}\rangle$ be a finite presentation, with all words in $R$ freely reduced, cyclically reduced, and distinct. Let $P_{t}^{k}$ be as before, with $k \geq 12$. Then
 \[
  \tor_{1}(\overline{P_{t}^{k}})=\llangle \overline{r_{1}t}, \ldots, \overline{r_{n}t}, \overline{t} \rrangle^{\overline{P_{t}^{k}}}
 \]
\end{lem}

\begin{proof}
(In this proof, we take all normal closures to be in $\overline{P_{t}^{k}}$.)

Clearly $\{\overline{r_{1}t}, \ldots, \overline{r_{n}t}, \overline{t} \}$ are all torsion elements in $\overline{P_{t}^{k}}$, and so we have that $\llangle \overline{r_{1}t}, \ldots, \overline{r_{n}t}, \overline{t} \rrangle  \subseteq \tor_{1}(\overline{P_{t}^{k}})$.

To show the converse, it suffices to show that $\tor(\overline{P_{t}^{k}})\subseteq \llangle \overline{r_{1}t}, \ldots, \overline{r_{n}t}, \overline{t} \rrangle$. So take some torsion element $g \in \tor(\overline{P_{t}^{k}})$ with $\ord(g)=n$. By Proposition \ref{powersB}, $(P_{t}^{k})_{sym}$ satisfies the $C'(1/6)$ small cancellation condition; thus, by Lemma \ref{C 1/6} $g$ is conjugate to some $\overline{s}$ with $s^{n}= r$ for some relator $r$ of $(P_{t}^{k})_{sym}$. If $r=t^{k}$ or $t^{-k}$ then $s$ is a power of $t$ and so $g$ is conjugated into $\llangle \overline{r_{1}t}, \ldots, \overline{r_{n}t}, \overline{t} \rrangle$.
\\Otherwise, $s^{n}$ is equal to some cyclic permutation of some $(r_{i}t)^{k}$ or $(r_{i}t)^{-k}$; it is enough to just consider the first case. Then, there is some cyclic permutation $q$ of $s$ such that $q^{n}= (r_{i}t)^{k}$ as words in $X^{*}$. 
\\What word $q$ can we have which satisfies  $(r_{i}t)^{k} = q^{n}$? If $|q|<|r_{i}|$, then $q$ contains no $t$ and thus $q^{n}$ contains no $t$; a contradiction. If $|q|=|r_{i}t|$ then $q=r_{i}t$ and so $g$ is conjugate to $\overline{q}=\overline{r_{i}t}$ which clearly lies in $\llangle \overline{r_{1}t}, \ldots, \overline{r_{n}t}, \overline{t} \rrangle$. If $|q|>|r_{i}t|$ then $q=(r_{i}t)^{z}a$, where $r_{i}=ab$ is a decomposition of $r_{i}$. Thus $q^{n}=((r_{i}t)^{z}a)^{n}$, and this can only be equal to $(r_{i}t)^{k}$ if $a=\emptyset$. In this case $q=(r_{i}t)^{z}$ for some $z$, and so $g$ is conjugate to $\overline{q}=\overline{(r_{i}t)^{z}}$ which lies in $\llangle \overline{r_{1}t}, \ldots, \overline{r_{n}t}, \overline{t} \rrangle$. Thus $\tor(\overline{P_{t}^{k}})\subseteq \llangle \overline{r_{1}t}, \ldots, \overline{r_{n}t}, \overline{t} \rrangle$.
\end{proof}

\begin{thm}\label{exact tor quot}
 Let $P$ be a finite presentation with all words in $R$ freely reduced, cyclically reduced, and distinct. Then, for any $k \geq 12$, $\overline{P_{t}^{k}}$ is word-hyperbolic, virtually special, and satisfies
 \[
\overline{P_{t}^{k}}/\tor_{1}(\overline{P_{t}^{k}})  \cong \overline{P} 
 \]
 Thus, in this case, $\torlen(\overline{P_{t}^{k}})=\torlen(\overline{P})+1$.
\end{thm}

\begin{proof}
This follows immediately from Proposition \ref{powersB} and Lemma \ref{clever torsion}.
\end{proof}

\begin{rem}
One may ask why the introduction of the extra generator $t$ is necessary when constructing $P_{t}^{k}$. It is indeed true that, given a finite presentation $P=\langle x_{1}, \ldots, x_{m} \ | \ r_{1}, \ldots, r_{n}\rangle$, the finite presentation $Q=\langle x_{1}, \ldots, x_{m} \ | \ r_{1}^{k_{1}}, \ldots, r_{n}^{k_{n}}\rangle$ (where  $k_{1}, \ldots, k_{n} \in \mathbb{N}_{\geq 1}$) presents a group $\overline{Q}$ with $\torlen(\overline{Q})-1\leq \torlen(\overline{P})\leq \torlen(\overline{Q})$. The reader can easily verify this. However, it is not necessarily the case that $\overline{Q}/\tor_{1}(\overline{Q}) \cong \overline{P}$. As an example, we can consider $P=\langle x,y,z | x, y^{3}, xy=z^{3} \rangle$ and $Q=\langle x,y,z | x^{3}, y^{3}, xy=z^{3} \rangle.$ It is clear that $P$ is just a presentation for the cyclic group with $9$ elements, $C_{9}$.  On the other hand, by \cite[Proposition 3.1]{ChiVya}, $\overline{Q}/\tor_{1}(\overline{Q}) \cong C_{3}.$
\end{rem}

\begin{thm}\label{fp hyp vspecialB}
 There is a finitely presented word-hyperbolic virtually special group $W$ with $\torlen(W) = \omega$. In particular, $W$ is virtually torsion-free.
\end{thm}

\begin{proof}
 Let $P$ be a finite presentation of a group with infinite torsion length; such things exist, by Theorem \ref{txfpg}. Then, by Theorem \ref{exact tor quot}, $\overline{P_{t}^{12}}$ is hyperbolic and virtually special, and $\torlen(\overline{P_{t}^{12}})=\torlen(\overline{P})+1=\omega$. Take $W=\overline{P_{t}^{12}}$.
\end{proof}

\begin{rem}\label{Bum remark}
The main construction in \cite{BumWis} can be used to obtain  a similar result to Theorem \ref{exact tor quot} above. Given a finite presentation $P=\langle A|R\rangle$,  we see in equation (4) of \cite[pp.~141]{BumWis} an explicit finite presentation of a $C'(1/6)$ group $G$ and $N\vartriangleleft G$ such that $G/N\cong \overline{P}$, and moreover that $\overline{P}$ is isomorphic to $\out(N)$ (\cite[Theorem 11]{BumWis}). Further analysis shows that $N=\tor_{1}(G)$, normally generated by $2$ elements. However,  both the finite presentation of $G$ in \cite{BumWis} and its manner of construction seem to be substantially more complicated than the finite presentation $P_{t}^{12}$ constructed above
\end{rem}

\begin{rem}
In \cite[Lemma 2.3]{ChiVya} we showed that $\tor_{i}(H) \leq \tor_{i}(G)$ whenever $H \leq G$. However, this does not extend to bounding torsion length of subgroups, even for finitely presented groups. Using the fact that there are finitely presented groups of any torsion length (\cite[Theorem 3.3]{ChiVya}), including $\omega$ (Theorem \ref{txfpg}), along with the Adian-Rabin construction (\cite[Theorem 2.4]{Chiodo}), one can show that given any finitely presented group $H$, and any ordinal $1\leq n \leq \omega,$ there is a finitely presented group of torsion length $n$ into which $H$ embeds. 
\end{rem}

We finish this section with an alternate construction for, and strengthening of, the result obtained as \cite[Theorem 3.3]{ChiVya}.

\begin{thm}\label{gen first paper}
 Define the sequence of finite presentations $P_{0}:=\langle - | - \rangle$, $P_{1}:=\langle t_{1} \ | \ t_{1}^{12}\rangle$, $P_{2}:=\langle t_{1}, t_{2} \ | \ (t_{1}^{12}t_{2})^{12}, t_{2}^{12}\rangle$, and, in general
 \[
 P_{n}:=\langle t_{1}, \ldots, t_{n} \ | \   (\cdots (t_{1}^{12}t_{2})^{12})\cdots t_{n})^{12}, \ldots, (t_{n-1}^{12}t_{n})^{12}, t_{n}^{12}\rangle
 \]
 Then $\overline{P}_{n}$ is $C'(1/6)$ (so word-hyperbolic and virtually special), $\overline{P}_{n}/\tor_{1}(\overline{P}_{n})\cong \overline{P}_{n-1}$, and $\torlen(\overline{P}_{n})=n$.
\end{thm}

\begin{proof}
 This follows immediately from Theorem \ref{exact tor quot}.
\end{proof}


\section{Quotients} \label{quotients}

We are interested in the universal torsion-free quotients of finitely presented groups. We begin with the following observation.

\begin{prop}
Let $G$ be a finitely presented group with infinite torsion length (see Theorem $\ref{fp hyp vspecialB}$). Then $G/\tor_{\omega}(G)$ is finitely generated and recursively presented, but \emph{not} finitely presented.
\end{prop}

\begin{proof}
 Assume $G/\tor_{\omega}(G)$ is finitely presented.  Then we have that $\tor_{\omega}(G)$ must be the normal closure of finitely many elements of $G$; say $\tor_{\omega}(G)$ $= \llangle g_{1}, \ldots, g_{n} \rrangle^{G}.$ But then each $g_{i}$ lies in some $\tor_{m_{i}}(G)$, and  as the  $\tor_{j}(G)$ form a  nested sequence we have that all the $g_{i}$ lie in $\tor_{M}(G)$ for $M=\max \{ m_{i} \}$. Thus $\tor_{\omega}(G) = \tor_{M}(G)$, and so $G$ has finite torsion length; a contradiction.
\end{proof}

With this in mind, we ask the following question:

\begin{ques}\label{our ques}
 Is there a finitely presented group $G$ for which $G/\tor_{1}(G)$ is recursively presented but \emph{not} finitely presented?
\end{ques}

Note that if such a group were to exist, then using the Adian-Rabin construction (\cite[Theorem 2.4]{Chiodo}) one could construct a group $G$ such that any sequence drawing from ``finitely presented''  and ``not finitely presented'' is realised by looking at the sequence $G/\tor_{1}(G)$, $G/\tor_{2}(G)$, $\ldots$. 

In the case of word-hyperbolic groups, however, it is always true that $G/\tor_{1}(G)$ is finitely presented, as we now show; moreover, in this context a finite presentation for $G/\tor_{1}(G)$ can be algorithmically constructed. We begin with a result of Papasoglu.

\begin{thm}[\cite{Papa}]\label{hyperbolic check}
There is a partial algorithm that, on input of a finite presentation $P$, halts if and only if $\overline{P}$ is a word-hyperbolic group. Moreover, when this algorithm does halt, it outputs a hyperbolicity constant $\delta$ for $P$.
\end{thm}

For a finitely generated group $G$ with finite generating set $X$, we define the ball of radius $r$ about the identity, $B_{X}(e,r)$, to be the set of elements
\[
B_{X}(e,r):= \{g \in G \ | \ \exists w \in X^{*} \textnormal{ with }|w|\leq r \textnormal{ and } \overline{w}=g \textnormal{ in } G\}.
\]

The following standard lemma will be of use; the proof of \cite[III.$\Gamma$ Theorem 3.2]{BriHae} provides an argument to verify it:

\begin{lem}\label{hyp fin2}
Let $G$ be a finitely presented word-hyperbolic group with hyperbolicity constant $\delta$. Then any finite subgroup $H \leq G$ is conjugate in $G$ to some subgroup in the $(4\delta + 2)$-ball around the origin. That is, there exists some $g\in G$ such that $g^{-1}Hg\subseteq B(e, 4\delta +2)$.
\end{lem}

\begin{thm}\label{hyp tor2}
Let $P=\langle X|R \rangle$ be a finite presentation of a word-hyperbolic group $G$ with hyperbolicity constant $\delta$. Let $S$ be the set
\[
 \{g \in \tor(G) \ | \  \langle g \rangle \subseteq B_{X}(e, 4\delta +2)\}
\]
Then $\llangle \tor(G) \rrangle^{G} = \llangle S \rrangle^{G}$. Moreover, from $P$ and $\delta$ we can explicitly compute the set $S$.
\end{thm}

\begin{proof}
Let $g$ be a torsion element in $G$. Then, by Lemma \ref{hyp fin2}, $\langle g\rangle$ is conjugate to a subgroup in the ball $B_{X}(e, 4\delta +2)$. Thus $\llangle \tor(G) \rrangle^{G} = \llangle S \rrangle^{G}$, and so the first statement is proved. 

Now, using the uniform solution to the word problem for hyperbolic groups (see \cite[III.$\Gamma$ Theorems 2.4--2.6]{BriHae}), we can identify a set of words (of length at most $r$) together representing all elements in $S$ as follows: enumerate all words of length at most $r$ in $X^{*}$; call these $w_{1}, \ldots, w_{k}$. For each $w_{i}$, compute minimal-length words for $w_{i}^{2}, w_{i}^{3},\ldots $  and so on until either some $\overline{w_{i}^{m}}$ lies outside $B_{X}(e, 4\delta +2)$ or is trivial. If, for $w_{i}$, the former occurs first, then discard $w_{i}$. If, for $w_{i}$, the latter occurs first, then add $w_{i}$ to our set. At the end of this process, we will have formed the set $S$, algorithmically from $P$ and $\delta$.
\end{proof}

Using Theorems \ref{exact tor quot} and \ref{hyp tor2}, we immediately see the following:

\begin{cor}
 Let $G$ be a finitely presented word-hyperbolic group. Then $G/\tor_{1}(G)$ is finitely presented. Moreover, \emph{any} finitely presented group $Q$ can be obtained as $Q \cong G/\tor_{1}(G)$ for $G$ some $C'(1/6)$ (and therefore word-hyperbolic) group $G$.
\end{cor}



\ 

\noindent \scriptsize{\textsc{Department of Pure Mathematics and Mathematical Statistics, University of Cambridge,
\\Wilberforce Road, Cambridge, CB3 0WB, UK. 
\\mcc56@cam.ac.uk
\vspace{5pt}
\\Department of Mathematics, Ben-Gurion University of the Negev,
\\P.O.B. 653, Beer-Sheva, 84105, Israel. 
\\vyas@math.bgu.ac.il}


\begin{thebibliography}{99}


\bibitem{Cohen2} S. Aanderaa, D. Cohen, \emph{Modular machines, the word problem for finitely presented groups and Collins' theorem}. Word problems, II (Conf.~on Decision Problems in Algebra, Oxford, 1976), pp.~1--16, 
Stud.~Logic Foundations Math., 95, North-Holland, Amsterdam-New York, 1980. 

\bibitem{Cohen3} S. Aanderaa, D. Cohen, \emph{Modular machines and the Higman-Clapham-Valiev embedding theorem}. Word problems, II (Conf.~on Decision Problems in Algebra, Oxford, 1976), pp.~7--28, 
Stud.~Logic Foundations Math., 95, North-Holland, Amsterdam-New York, 1980. 

\bibitem{Agol} I. Agol, \emph{The virtual Haken conjecture},  Doc. Math. \textbf{18}, 1045--1087 (2013). With an appendix by Agol, Daniel Groves, and Jason Manning.


\bibitem{BerHil} A. Berrick, J. Hillman, \emph{Perfect and acyclic subgroups of finitely presentable groups},  J. London Math. Soc. (2) \textbf{68}, no. 3, 683--698 (2003).

\bibitem{BriHae} M. Bridson, A. Haefliger, \emph{Metric Spaces of Non-Positive Curvature}, Springer, (1999).


\bibitem{BrodHow} S. D. Brodsky, J. Howie, \emph{The universal torsion-free image of a group}, Israel J. Math. \textbf{98}, 209--228 (1997).


\bibitem{BumWis} I. Bumagin and D.T. Wise, \emph{Every group is an outer automorphism group of a finitely generated group}, J. Pure App. Alg. \textbf{200}, no. 1, 137--147 (2005). 

\bibitem{Chiodo} M. Chiodo, \emph{Finding non-trivial elements and splittings in groups}, J. Algebra. \textbf{331}, 271--284 (2011).

\bibitem{Chiodo3} M. Chiodo, \emph{On torsion in finitely presented groups}, Groups Complex. Cryptol. (2) \textbf{6}, 1--8 (2014).

\bibitem{ChiVya} M.~Chiodo, R.~Vyas, \emph{A note on torsion length}, Comm.~Algebra, \textbf{43}, no.~11,  4825--4835 (2015).

\bibitem{Cirio et al} L. Cirio, A. D'Andrea, C. Pinzari, S. Rossi, \emph{Connected components of compact matrix quantum groups and finiteness conditions}, J. Funct. Anal. \textbf{267}, no. 9, 3154--3204 (2014).


\bibitem{Cohen} D. Cohen, \emph{Combinatorial group theory: A topological approach (London Mathematical Society Student Texts)}, Cambridge University Press, Cambridge, (1989).



\bibitem{Gre} M. Greendlinger, \emph{On Dehn's algorithms for the conjugacy and word problems, with applications}, Comm. Pure Appl. Math. \textbf{13}, 641--677 (1960).


\bibitem{Hig emb} G. Higman, \emph{Subgroups of finitely presented groups}, Proc. Royal Soc. London Ser. A \textbf{262}, 455--475 (1961).



\bibitem{KaMaSo} A. Karrass, W. Magnus, D. Solitar, \emph{Elements of finite order in groups with a single defining relation}, Comm. Pure Appl. Math. \textbf{13}, 57--66 (1960).

\bibitem{KarSol} A. Karrass, D. Solitar, \emph{One relator groups having a finitely presented normal subgroup},  
Proc. Amer. Math. Soc. \textbf{69}, no. 2, 219--222 (1978). 

\bibitem{Papa} P. Papasoglu, \emph{An algorithm detecting hyperbolicity}. Geometric and computational perspectives on infinite groups (Minneapolis, MN and New Brunswick, NJ, 1994), DIMACS Ser. Discrete Math. Theoret. Comput. Sci., \textbf{25}, Amer. Math. Soc., Providence, RI, 193--200 (1996).


\bibitem{Rot} J. Rotman, \emph{An introduction to the theory of groups; fourth edition}, Springer-Verlag, New York, (1995).

\bibitem{Sch} P. Schupp, \emph{A survey of small cancellation theory},
in Word Problems, eds. Boone, Cannonito, and Lyndon, Amsterdam, North-Holland, 569--589 (1973).

\bibitem{Simpson} S. Simpson, \emph{A slick proof of the unsolvability of the word problem for finitely presented groups}, www.personal.psu.edu/t20/logic/seminar/050517.pdf, 5 pages, unpublished online notes.

\bibitem{Wise}  D. Wise, \emph{Cubulating small cancellation groups}, Geom. Funct. Anal. \textbf{14}, no.~1,  150--214 (2004).


\end{thebibliography}
\end{document}